\documentclass[conference]{IEEEtran}
\IEEEoverridecommandlockouts
\usepackage{cite}
\usepackage{amsmath,amssymb,amsfonts,amsthm}
\usepackage{algorithmic}
\usepackage{graphicx}
\usepackage{textcomp}
\usepackage{xcolor}
\usepackage{latexsym,bm}
\usepackage{setspace}
\def\BibTeX{{\rm B\kern-.05em{\sc i\kern-.025em b}\kern-.08em
    T\kern-.1667em\lower.7ex\hbox{E}\kern-.125emX}}
\newtheorem{thm}{Theorem}
\newtheorem{pro}{Property}
\newtheorem{defn}{Definition}

\newtheorem{corollary}{Corollary}

\usepackage{geometry}
\usepackage{setspace}
\begin{document}

\title{Convolution and correlation theorems for the windowed offset linear canonical transform*\\
\thanks{This work is supported by the National Natural Science Foundation of China (No. 61671063), and also by the Foundation
for Innovative Research Groups of the National Natural Science Foundation of China (No. 61421001).}
}

\author{\IEEEauthorblockN{ Wen-Biao, Gao}
\IEEEauthorblockA{\textit{School of Mathematics and Statistics} \\
\textit{Beijing Institute of Technology}\\
Beijing 102488, China \\
}
\and
\IEEEauthorblockN{ Bing-Zhao, Li}
\IEEEauthorblockA{\textit{Beijing Key Laboratory on MCAACI} \\
\textit{Beijing Institute of Technology}\\
Beijing 102488, China \\
Email address:li$\_$bingzhao@bit.edu.cn}
}

\maketitle

\begin{abstract}
In this paper, some important properties of the windowed offset linear canonical transform (WOLCT) such as shift, modulation and
orthogonality relation are introduced. Based on these properties we derive the convolution and correlation theorems
for the WOLCT.
\end{abstract}


\section{Introduction}
The offset linear canonical transform (OLCT) \cite{b1,b2,b4,b5,b12,b13} is a generalized version of the linear canonical transform (LCT) with four parameters $(a, b, c, d)$ \cite{b6,b7,b8,b9,b10,b11}. It is a six parameter $(a, b, c, d, u_{0}, w_{0})$ class of linear integral transform. Because of the two extra parameters, time shifting $u_{0}$ and frequency modulation $w_{0}$, the OLCT are more general and flexible than the LCT. It has been widely applied in signal processing and optics \cite{b7,b8,b9,b11}.

As a mathematical operation, the convolution provides some applications in pure and applied mathematics such as numerical linear algebra, numerical analysis, and signal processing \cite{b18,b19}. Correlation is similar to convolution and it is another useful operation in signal
processing, optics and detection applications \cite{b20,b21,b22}. In some domains such as the LCT domain \cite{b6,b3,b13}, Wignar-Ville transform domain \cite{b8,b12,b1} and the OLCT domain \cite{b12,b1}, the convolution and correlation operations have been studied. They are the most fundamental and important theorems in these domains.

In \cite{b14,b15,b17}, the window function based on linear canonical transform (WLCT) were presented. It is believed to be a new and important signal processing tool. They obtained some main properties such as covariance property, shift, modulation, orthogonality property and inversion formulas.
The results of the WLCT have been well applied. Although the windowed offset linear canonical transform (WOLCT)\cite{b23} has been proposed, some properties of WOLCT has not been studied. The purpose of this paper is to study the WOLCT. Some properties of WOLCT are obtained and its convolution and correlation theorems are derived. The results are very important application in some fields such as digital signal and image processing.

In this paper we first review the OLCT. Next, we introduce the definition of the WOLCT, and obtain some important properties such as linearity, inversion formula and parity. Finally, we present the convolution and correlation theorems for the WOLCT.

\section{Preliminary.}

Let us briefly review some basic properties.

The Lebesgue space $L^{2}(\mathbb{R})$ is defined as the space of all measurable functions on $\mathbb{R}$ such that
\begin{align}
		\|f\|_{L^{2}(\mathbb{R})}=\left(\int_{\mathbb{R}}|f(t)|^{2}\rm{d}\mit t\right)^{\frac{1}{2}}<\infty
	\end{align}
Now we introduce an inner product of the functions $f,g$ defined on $L^2(\mathbb{R})$ is given by
\begin{align}
	\langle f,g\rangle_{L^2(\mathbb{R})}=\int_{\mathbb{R}}f(t)\overline{g(t)}{\rm d}\mit t.
\end{align}
\begin{defn}(OLCT) \cite{b1}
Let $A=(a,b,c,d,u_{0},w_{0})$ be a matrix parameter satisfying $a,b,c,d,u_{0},w_{0}\in \mathbb{R}$, and $ad-bc=1$. The OLCT of a signal $f(t)\in L^{2}(\mathbb{R})$ is defined by
\begin{align}
		\begin{split}
			O_{A}f(u)=O_{A}[f(t)](u)=\begin{cases}
		\int_{-\infty}^{+\infty}f(t)K_{A}(t,u)\rm{d}\mit t,   &b\neq0  \\
		\sqrt{d}e^{i\frac{cd}{2}(u-u_{0})^{2}+iuw_{0}}\\
 \times f(d(u-u_{0})),    &b=0
		\end{cases}
		\end{split}
	\end{align}
where
\begin{align}
\begin{split}
	K_{A}(t,u)&=\frac{1}{\sqrt{i2\pi b}}e^{i\frac{a}{2b}t^{2}-i\frac{1}{b}t(u-u_{0})-i\frac{1}{b}u(du_{0}-bw_{0})}\\
&\times e^{i\frac{d}{2b}(u^{2}+u_{0}^{2})}
\end{split}
\end{align} \end{defn}
From definition 1 it can be seen that for case $b = 0$ the OLCT is simply a time scaled version off multiplied by a linear chirp. Hence, without loss of generality, we assume $b\neq 0$. If $u_{0}=0$ and $w_{0}=0$, the OLCT reduces to the LCT \cite{b1,b3,b5,b7}.

The inverse of an OLCT with parameters $A=(a,b,c,d,u_{0},w_{0})$ is given by an OLCT with parameters $A^{-1}=(d,-b,-c,a,bw_{0}-du_{0},cu_{0}-aw_{0})$. The exact inverse OLCT expression is \cite{b16}
\begin{align}
\begin{split}
	f(t)=O_{A^{-1}}(O_{A}f(u))(t)&=e^{i\frac{cd}{2}u_{0}^{2}-iadu_{0}w_{0}+i\frac{ab}{2}w_{0}}\\
&\times \int_{-\infty}^{+\infty}O_{A}f(u)K_{A^{-1}}(u,t)\rm{d}\mit u,
\end{split}
\end{align}
Next, we introduce one of important properties for the OLCT, its generalized
Parseval formula \cite{b4,b5}, as follows:
\begin{align}
	\int _{\mathbb{R}}f(t)\overline{g(t)}\rm{d}\mit t=\int _{\mathbb{R}}O_{A}\mit f(u))\overline{O_{A}\mit g(u))}\rm{d} \mit u
\end{align}
\section{The windowed offset linear canonical transform}
Here we give definition of the WOLCT, then introduce some properties.
\begin{defn} (WOLCT) \cite{b23} Let $\phi\in L^2(\mathbb{R})\backslash \{0\}$ be a window function. The WOLCT of a signal $f\in L^2(\mathbb{R})$ with respect to $\phi$ is defined by
\begin{align}
	V^{A}_{\phi}f(u,w)=\int_{\mathbb{R}}f(t) \overline{\phi(t-w)} K_{A}(t,u)\rm{d}\mit t
	\end{align}	
where $K_{A}(t,u)$ is given by (4).
\end{defn}
For a fixed $w$, we have
\begin{align}
		V^{A}_{\phi}f(u,w)=O_{A}[f(t)\overline{\phi(t-w)}](u)
	\end{align}
Using the inverse OLCT to (8), we have
\begin{align}
		\begin{split}
		f(t)\overline{\phi(t-w)}&=O_{A^{-1}}(V^{A}_{\phi}f(u,w))(t)\\
&=e^{i\frac{cd}{2}u_{0}^{2}-iadu_{0}w_{0}+i\frac{ab}{2}w_{0}}\\
&\times\int_{-\infty}^{+\infty}V^{A}_{\phi}f(u,w)K_{A^{-1}}(u,t)\rm{d}\mit u,
		\end{split}
	\end{align}
Some basic properties of the WOLCT are summarized in the following theorem.
\begin{pro}[Linearity]
 Let $\phi\in L^2(\mathbb{R})\backslash \{0\}$ be a window function and $f,g\in L^2(\mathbb{R})$ the WOLCT is a linear operator, namely,
  \begin{align}
		\left[V^{A}_{\phi}(\lambda f+\mu g)\right](u,w)= \lambda V^{A}_{\phi}f(u,w)+\mu V^{A}_{\phi}g(u,w)
	\end{align}
for arbitrary constants $\lambda$ and $\mu$.
\end{pro}
\begin{proof} This follows directly from the linearity of the product and the integration
involved in Definition 2. \end{proof}
\begin{pro}[Shift]
 Let $\phi\in L^2(\mathbb{R})\backslash \{0\}$ be a window function and $f\in L^2(\mathbb{R})$. Then we have
\begin{align}
        \begin{split}
		V^{A}_{\phi}\{T_{t_{0}}f\}(u,w)&= V^{A}_{\phi}\{f\}(u-at_{0},w-t_{0})\\
&\times e^{iat_{0}w_{0}-i\frac{ac}{2}t_{0}^{2}+ict_{0}(u-u_{0})}\\
		\end{split}
	\end{align}
where $T_{t_{0}}f(t)=f(t-t_{0})$. \end{pro}
\begin{proof} By Definition 2, we have
\begin{align}
		V^{A}_{\phi}\{T_{t_{0}}f\}(u,w)&= \int_{\mathbb{R}}f(t-t_{0}) \overline{\phi(t-w)}K_{A}(t,u)\rm{d}\mit t
	\end{align}
By making the change of variable $x=t-t_{0}$ in the above expression, we obtain
\begin{align}
        \begin{split}
		V^{A}_{\phi}\{T_{t_{0}}f\}(u,w)&= \int_{\mathbb{R}}f(x) \overline{\phi(x-(w-t_{0}))}\\
&\times K_{A}(x+t_{0},u)\rm{d}\mit x\\
&=\int_{\mathbb{R}}f(x) \overline{\phi(x-(w-t_{0}))}Y\rm{d}\mit x\\
&=\int_{\mathbb{R}}f(x) \overline{\phi(x-(w-t_{0}))}\\
&\times K_{A}(x,u-at_{0})\rm{d}\mit x\\
&\times e^{iat_{0}w_{0}-i\frac{a}{2b}t_{0}^{2}(ad-1)+i\frac{1}{b}t_{0}(u-u_{0})(ad-1)}\\
&=V^{A}_{\phi}\{f\}(u-at_{0},w-t_{0})\\
&\times e^{iat_{0}w_{0}-i\frac{ac}{2}t_{0}^{2}+ict_{0}(u-u_{0})}\\
		\end{split}
	\end{align}
where $Y=\frac{1}{\sqrt{i2\pi b}} e^{i\frac{a}{2b}x^{2}+i\frac{a}{2b}t_{0}^{2}+i\frac{a}{b}xt_{0}-i\frac{1}{b}x(u-u_{0})}\\
\times e^{-i\frac{1}{b}t_{0}(u-u_{0})
-i\frac{1}{b}u(du_{0}-bw_{0})+i\frac{d}{2b}(u^{2}+u_{0}^{2})}$.

Which completes the proof.\end{proof}
\begin{pro}[Modulation]
 Let $\phi\in L^2(\mathbb{R})\backslash \{0\}$ be a window function and $f\in L^2(\mathbb{R})$. Then we have
\begin{align}
        \begin{split}
		V^{A}_{\phi}\{M_{s}f\}(u,w)&=V^{A}_{\phi}\{f\}(u-bs,w)\\
&\times e^{ibsw_{0}-i\frac{db}{2}s^{2}+ids(u-u_{0})}
\end{split}
	\end{align}
where $M_{s}f(t)=f(t)e^{ist}$. \end{pro}
\begin{proof}  From Definition 2, it follows that
\begin{align}
        \begin{split}
V^{A}_{\phi}\{M_{s}f\}(u,w)&=\int_{\mathbb{R}}f(t)e^{ist} \overline{\phi(t-w)}K_{A}(t,u)\rm{d}\mit t\\
&=\int_{\mathbb{R}}f(t)\overline{\phi(t-w)}Y_{1}\rm{d}\mit t\\
&=\int_{\mathbb{R}}f(t)\overline{\phi(t-w)}K_{A}(t,u-bs)\rm{d}\mit t\\
&\times e^{ibsw_{0}-i\frac{db}{2}s^{2}+ids(u-u_{0})}\\
&=V^{A}_{\phi}\{f\}(u-bs,w)\\
&\times e^{ibsw_{0}-i\frac{db}{2}s^{2}+ids(u-u_{0})}\\
\end{split}
	\end{align}
where $Y_{1}=\frac{1}{\sqrt{i2\pi b}} e^{i\frac{a}{2b}t^{2}-i\frac{1}{b}t((u-bs)-u_{0})-i\frac{1}{b}u(du_{0}-bw_{0})}\\
e^{i\frac{d}{2b}(u^{2}+u_{0}^{2})}$.

Which completes the proof.\end{proof}

\begin{pro}[Shift and modulation property]
 Let $\phi\in L^2(\mathbb{R})\backslash \{0\}$ be a window function and $f\in L^2(\mathbb{R})$. Then we have
\begin{align}
		EV^{A}_{\phi}\{T_{t_{0}}M_{s}f\}(u,w)
=V^{A}_{\phi}\{f\}(u-bs-at_{0},w-t_{0})
	\end{align}
where $E=e^{i\frac{1}{b}(bs+at_{0})(d(u_{0}-u)+\frac{b}{2}(ds+ct_{0})-bw_{0})}\\
\times e^{i\frac{t_{0}}{2b}(2u-bs-2u_{0})}$.
\end{pro}
\begin{proof}  The proof of Property 4 can be achieved by directly combining Properties 2 and Properties 3.
\end{proof}

We will prove the inversion formula for the WOLCT by the connection between the OLCT and the WOLCT. From this theorem, we know that it is possible to restore the original signal $f$ perfectly using the inverse WOLCT as follows.
\linespread{0.05}
\begin{pro}[Inversion formula]
 Let $\phi,\psi\in L^2(\mathbb{R})\backslash \{0\}$ be window function, $\langle\psi,\phi \rangle\neq0$ and $f\in L^2(\mathbb{R})$. Then we get
 \begin{align}
        \begin{split}
		f(t)&=\frac{1}{\langle\psi,\phi \rangle}e^{i\frac{cd}{2}u_{0}^{2}-iadu_{0}w_{0}+i\frac{ab}{2}w_{0}}\int_{\mathbb{R}^{2}}V^{A}_{\phi}f(u,w)\\
&\times K_{A^{-1}}(u,t)\psi(t-w)\rm{d}\mit u\rm{d}\mit w\\
\end{split}
	\end{align}
\end{pro}
\begin{proof}
Multiplying both sides of (9) from the right by $ \psi(t-w)$ and integrating with respect to $ \rm{d}w$ we get
\begin{align}
		\begin{split}
		\int_{\mathbb{R}}f(t)\overline{\phi(t-w)}\psi(t-w)\rm{d}\mit w&=
e^{i\frac{cd}{2}u_{0}^{2}-iadu_{0}w_{0}+i\frac{ab}{2}w_{0}}\\
&\times\int_{\mathbb{R}^{2}}V^{A}_{\phi}f(u,w)K_{A^{-1}}(u,t)\\
&\times\psi(t-w)\rm{d}\mit u\rm{d}\mit w
		\end{split}
	\end{align}
Using (2), we have
\begin{align}
        \begin{split}
		f(t)&=\frac{1}{\langle\psi,\phi \rangle}e^{i\frac{cd}{2}u_{0}^{2}-iadu_{0}w_{0}+i\frac{ab}{2}w_{0}}\int_{\mathbb{R}^{2}}V^{A}_{\phi}f(u,w)\\
&\times K_{A^{-1}}(u,t)\psi(t-w)\rm{d}u\rm{d}\mit w\\
\end{split}
	\end{align}
which completes the proof.\end{proof}
If $\phi=\psi$, then
\begin{align}
        \begin{split}
        f(t)&=\frac{1}{\|\phi\|^{2}}e^{i\frac{cd}{2}u_{0}^{2}-iadu_{0}w_{0}+
        i\frac{ab}{2}w_{0}}\int_{\mathbb{R}^{2}}V^{A}_{\phi}f(u,w)\\
        &\times K_{A^{-1}}(u,t)\phi(t-w)\rm{d}\mit u\rm{d}\mit w\\
\end{split}
	\end{align}
\begin{pro}[Orthogonality relation for WOLCT]
 Let $\phi,\psi\in L^2(\mathbb{R})\backslash \{0\}$ be window function and $f,g\in L^2(\mathbb{R})$. Then we get
 \begin{align}
		\langle V^{A}_{\phi}f(u,w), V^{A}_{\psi}g(u,w)\rangle=\langle f,g \rangle\langle\psi,\phi\rangle
	\end{align}
\end{pro}
\begin{proof}
From (6) and (8), we get
\begin{align}
		\begin{split}
		\langle V^{A}_{\phi}f(u,w), V^{A}_{\psi}g(u,w)\rangle&=\int_{\mathbb{R}}\int_{\mathbb{R}}V^{A}_{\phi}f(u,w)\\
&\times\overline{V^{A}_{\psi}g(u,w)}\rm{d}\mit u\rm{d}\mit w\\
&=\int_{\mathbb{R}}\int_{\mathbb{R}}O_{A}[f(t)\overline{\phi(t')}](u)\\
&\times\overline{O_{A}[g(t)\overline{\psi(t')}](u)}\rm{d}\mit u\rm{d}\mit w\\
&=\int_{\mathbb{R}}\int_{\mathbb{R}}f(t)\overline{\phi(t')}(\overline{g(t)}\psi(t'))\rm{d}\mit t\rm{d}\mit w\\
&=\int_{\mathbb{R}}f(t)\overline{g(t)}(\int_{\mathbb{R}}\psi(t')\overline{\phi(t')}\rm{d}\mit w)\rm{d}\mit t\\
&=\langle f,g \rangle\langle\psi,\phi\rangle
		\end{split}
	\end{align}
where $t'=t-w$
which completes the proof.\end{proof}
Based on the above theorem, we may conclude the following important consequences.

(i) If $\phi=\psi$, then
\begin{align}
		\langle V^{A}_{\phi}f(u,w), V^{A}_{\phi}g(u,w)\rangle=\langle f,g \rangle\|\phi\|^{2}
	\end{align}

(ii) If $f=g$, then
\begin{align}
		\langle V^{A}_{\phi}f(u,w), V^{A}_{\psi}f(u,w)\rangle=\|f\|^{2}\langle\psi,\phi\rangle
	\end{align}

(iii) If  $f=g$ and $\phi=\psi$, then
\begin{align}
		\begin{split}
		\langle V^{A}_{\phi}f(u,w), V^{A}_{\phi}f(u,w)\rangle=\|f\|^{2}\|\phi\|^{2}\\
=\int_{\mathbb{R}}\int_{\mathbb{R}}|V^{A}_{\phi}f(u,w)|^{2}\rm{d}\mit u\rm{d}\mit w
		\end{split}
	\end{align}
\begin{pro}[Parity]
 Let $\phi\in L^2(\mathbb{R})\backslash \{0\}$ be a window function and $f\in L^2(\mathbb{R})$. Then we have
\begin{align}
		V^{A}_{P\phi}\{Pf\}(u,w)=V^{A}_{\phi}\{f\}(2u_{0}-u,-w)e^{i2w_{0}(u-u_{0})}
	\end{align}
where $Pf(t)=f(-t)$. \end{pro}

\begin{proof}  From Definition 2, it follows that
\begin{align}
        \begin{split}
V^{A}_{P\phi}\{Pf\}(u,w)&=\int_{\mathbb{R}}f(-t)\overline{\phi(-(t-w))}K_{A}(t,u)\rm{d}\mit t\\
&=\int_{\mathbb{R}}f(-t)\overline{\phi(-t-(-w))}K'_{A}\rm{d}\mit t\\
&\times e^{i\frac{1}{b}(2u_{0}-u)(du_{0}-bw_{0})-i\frac{d}{2b}((2u_{0}-u)^{2}+u_{0}^{2})}\\
&\times e^{-i\frac{1}{b}u(du_{0}-bw_{0})+i\frac{d}{2b}(u^{2}+u_{0}^{2})}\\
&=V^{A}_{\phi}\{f\}(2u_{0}-u,-w)e^{i2w_{0}(u-u_{0})}\\
\end{split}
	\end{align}
where $K'_{A}=\frac{1}{\sqrt{i2\pi b}}e^{i\frac{a}{2b}(-t)^{2}-i\frac{1}{b}(-t)((2u_{0}-u)-u_{0})}\\
\times e^{-i\frac{1}{b}(2u_{0}-u)(du_{0}-bw_{0})+i\frac{d}{2b}((2u_{0}-u)^{2}+u_{0}^{2})}$.

Which completes the proof.\end{proof}
\begin{pro}
 Let $\phi\in L^2(\mathbb{R})\backslash \{0\}$ be a window function and $f\in L^2(\mathbb{R})$. Then we have
\begin{align}
        \begin{split}
		V^{A}_{\overline{\phi}}\{\overline{f}\}(u,w)&=V^{A}_{f}\{\phi\}(u-aw,-w)\\
&\times e^{icw(u_{0}-u)+iaww_{0}-i\frac{ac}{2}w^{2}}
\end{split}
	\end{align} \end{pro}
\begin{proof}  From Definition 2, let $t-w=t_{1}$, then
\begin{align}
        \begin{split}
V^{A}_{\overline{\phi}}\{\overline{f}\}(u,w)&=\int_{\mathbb{R}}\overline{f(t)}\phi(t-w)K_{A}(t,u)\rm{d}\mit t\\
&=\int_{\mathbb{R}}\phi(t_{1})\overline{f(t_{1}-(-w))}\frac{1}{\sqrt{i2\pi b}}\\
&\times e^{i\frac{a}{2b}(t_{1}+w)^{2}-i\frac{1}{b}(t_{1}+w)((2u_{0}-u)-u_{0})}\\
&\times e^{-i\frac{1}{b}u(du_{0}-bw_{0})+i\frac{d}{2b}(u^{2}+u_{0}^{2})}\rm{d}\mit t_{1}\\
&= \int_{\mathbb{R}}\phi(t_{1})\overline{f(t_{1}-(-w))}\frac{1}{\sqrt{i2\pi b}}\\
&\times e^{i\frac{a}{2b}t_{1}^{2}-i\frac{1}{b}t_{1}((u-aw)-u_{0})}\\
&\times e^{-i\frac{1}{b}(u-aw)(du_{0}-bw_{0})+i\frac{d}{2b}((u-aw)^{2}+u_{0}^{2})}\rm{d}\mit t_{1}\\
&\times e^{i\frac{1}{b}(u-aw)(du_{0}-bw_{0})-i\frac{d}{2b}((u-aw)^{2}+u_{0}^{2})}\\
&\times e^{i\frac{a}{2b}w^{2}-i\frac{1}{b}w(u-u_{0})}\\
&\times e^{-i\frac{1}{b}u(du_{0}
-bw_{0})+i\frac{d}{2b}(u^{2}+u_{0}^{2})}\\
&=V^{A}_{f}\{\phi\}(u-aw,-w)\\
&\times e^{icw(u_{0}-u)+iaww_{0}-i\frac{ac}{2}w^{2}}\\
\end{split}
	\end{align}
which completes the proof.\end{proof}
\section{Convolution and Correlation Theorems for WOLCT}

In this section, we derive convolution and correlation theorems for WOLCT.

\begin{defn} [OLCT Convolution]
For any $f,g\in L^2(\mathbb{R})$, we define the convolution operation $\star$ for OLCT by
\begin{align}
	(f\star g)(t)=\int_{\mathbb{R}}f(x)g(t-x)e^{-i\frac{a}{2b}x(t-x)}\rm{d}\mit x
	\end{align}	
\end{defn}
As a consequence of the above definition, we get the following important theorem:
\begin{thm} [WOLCT Convolution]
Let $\phi\in L^{2}(\mathbb{R})\backslash \{0\}$. Then, for every $f,g\in L^{2}(\mathbb{R})$,
we have
\begin{align}
		\begin{split}
	V_{\phi\star\psi}^{A}(f\star g)(u,w)&=B
\int_{\mathbb{R}}V_{\phi}^{A}f(m_{0},m)V_{\psi}^{A}g(m_{1},w-m)\\
&\times e^{i\frac{a}{2b}m(\frac{da}{2}-1)(m-w)}\rm{d}\mit m\\
		\end{split}
	\end{align}
where $m_{0}=u-\frac{a}{2}(w-m)$, $m_{1}=u-\frac{a}{2}m$, $B=\sqrt{i2\pi b}e^{i\frac{1}{b}(u-\frac{a}{2}w)(du_{0}-bw_{0})}e^{-i\frac{da}{2b}w(\frac{a}{4}w-u)-i\frac{a}{2b}(u^{2}+u_{0}^{2})}$.
\end{thm}
\begin{proof}
Based on Definition 2 and Definition 3, we get
\begin{align}
		\begin{split}
	V_{\phi\star\psi}^{A}(f\star g)(u,w)&=\int_{\mathbb{R}}(f\star g)(t)(\overline{\phi}\star\overline{\psi})(t-w)K_{A}(t,u)\rm{d}\mit t\\
&=\int_{\mathbb{R}}\int_{\mathbb{R}}f(x)g(t-x)e^{-i\frac{a}{2b}x(t-x)}\rm{d}\mit x\\
&\times\int_{\mathbb{R}}\overline{\phi}(r)\overline{\psi}(t-w-r)e^{-i\frac{a}{2b}r(t-w-r)}\rm{d}\mit r\\
&\times K_{A}(t,u)\rm{d}\mit t\\
		\end{split}
	\end{align}
Setting $x=x_{1}, t=x_{1}+x_{2}, r=x_{1}-m$, \\
$y_{0}=i\frac{1}{b}u(du_{0}-bw_{0})$, $y_{1}=i\frac{d}{2b}(u^{2}+u_{0}^{2})$, we get
\begin{align}
		\begin{split}
	V_{\phi\star\psi}^{A}(f\star g)(u,w)&=\int_{\mathbb{R}^{3}}f(x_{1})g(x_{2})\overline{\phi}(x_{1}-m)\\
&\times\overline{\psi}(x_{2}-(w-m))
e^{-i\frac{a}{b}x_{1}x_{2}}\\
&\times e^{+i\frac{a}{2b}x_{1}(w-m)+i\frac{a}{2b}mx_{2}-i\frac{a}{2b}m(w-m)}\\
&\times \frac{1}{\sqrt{i2\pi b}}e^{i\frac{a}{2b}(x_{1}+x_{2})^{2}-i\frac{1}{b}(x_{1}+x_{2})(u-u_{0})}\\
&\times e^{-y_{0}+y_{1}}\rm{d}\mit x_{1}\rm{d}\mit x_{2}\rm{d}\mit m\\
&=\int_{\mathbb{R}^{3}}f(x_{1})g(x_{2})\overline{\phi}(x_{1}-m)\\
&\times\overline{\psi}(x_{2}-(w-m))\frac{1}{\sqrt{i2\pi b}}\\
&\times e^{i\frac{a}{2b}x_{1}^{2}-i\frac{1}{b}x_{1}(u-\frac{a}{2}(w-m)-u_{0})+i\frac{a}{2b}x_{2}^{2}}\\
&\times e^{-i\frac{1}{b}x_{2}(u-\frac{am}{2}-u_{0})-i\frac{a}{2b}m(w-m)}\\
&\times e^{-y_{0}+y_{1}}\rm{d}\mit x_{1}\rm{d}\mit x_{2}\rm{d}\mit m\\
&=B\int_{\mathbb{R}}\int_{\mathbb{R}}f(x_{1})\overline{\phi}(x_{1}-m)\\
&\times K_{A}(x_{1},m_{0})\rm{d}\mit x_{1}\\
&\times\int_{\mathbb{R}}g(x_{2})\overline{\psi}(x_{2}-(w-m))\\
&\times K_{A}(x_{2},m_{1})\rm{d}\mit x_{2} e^{i\frac{a}{2b}m(\frac{da}{2}-1)(m-w)}\rm{d}\mit m\\
&=B\int_{\mathbb{R}}V_{\phi}^{A}f(m_{0},m)V_{\psi}^{A}g(m_{1},w-m)\\
&\times e^{i\frac{a}{2b}m(\frac{da}{2}-1)(m-w)}\rm{d}\mit m \\
		\end{split}
	\end{align}
where $K_{A}(x_{1},m_{0})=\frac{1}{\sqrt{i2\pi b}}e^{i\frac{a}{2b}x_{1}^{2}-i\frac{1}{b}x_{1}(u-\frac{a}{2}(w-m)-u_{0})}\times\\
e^{-i\frac{1}{b}(u-\frac{a}{2}(w-m))(du_{0}-bw_{0})+i\frac{d}{2b}((u-\frac{a}{2}(w-m))^{2}+u_{0}^{2})}$,\\
$K_{A}(x_{2},m_{1})=\frac{1}{\sqrt{i2\pi b}}e^{i\frac{a}{2b}x_{2}^{2}-i\frac{1}{b}x_{2}(u-\frac{a}{2}m-u_{0})}\times\\
e^{-i\frac{1}{b}(u-\frac{a}{2}m)(du_{0}-bw_{0})+i\frac{d}{2b}((u-\frac{a}{2}m)^{2}+u_{0}^{2})}$.  \end{proof}

\setlength{\lineskip}{0.5em}
\begin{corollary}
If the parameter of WOLCT changes to $(a,b,c,d,u_{0},w_{0})=(a,b,c,d,0,0)$, then the Theorem 1 reduces to convolution theorem as follows:
\begin{align}
		\begin{split}
	V_{\phi\star\psi}^{A}(f\star g)(u,w)&=\sqrt{i2\pi b}e^{-i\frac{da}{2b}w(\frac{a}{4}w-u)-i\frac{a}{2b}u^{2}}\\
&\times\int_{\mathbb{R}}G_{\phi}^{A}f(m_{0},m)G_{\psi}^{A}g(m_{1},w-m)\\
&\times e^{i\frac{a}{2b}m(\frac{da}{2}-1)(m-w)}\rm{d}\mit m\\
		\end{split}
	\end{align}
where $G_{\phi}^{A}f$ and $G_{\psi}^{A}g$ are the window functions in the LCT domain of $f$ and $g$ \cite{b14,b15}, respectively.
\end{corollary}
\begin{corollary}
If the parameter of WOLCT changes to $(a,b,c,d,u_{0},w_{0})=(0,1,-1,0,0,0)$, then the Theorem 1 reduces to convolution theorem as follows:
\begin{align}
		\begin{split}
	V_{\phi\star\psi}^{A}(f\star g)(u,w)&=\sqrt{i2\pi b}e^{-iuw_{0}}
\int_{\mathbb{R}}V_{\phi}^{A}f(u,m)\\
&\times V_{\psi}^{A}g(u,w-m)\rm{d}\mit m
		\end{split}
	\end{align}
\end{corollary}
\begin{defn} [OLCT Correlation]
For any $f,g\in L^2(\mathbb{R})$, we define the correlation operation $\circ$ for OLCT by
\begin{align}
	(f\circ g)(t)=\int_{\mathbb{R}}\overline{f(x)}g(x+t)e^{i\frac{a}{2b}x(x+t)}\rm{d}\mit x
	\end{align}	
\end{defn}
Next, we establish the correlation theorem for the WOLCT.
\begin{thm} [WOLCT Correlation]
Let $\phi\in L^{2}(\mathbb{R})\backslash \{0\}$. Then, for every $f,g\in L^{2}(\mathbb{R})$,
we have
\begin{align}
		\begin{split}
	V_{\phi\circ\psi}^{A}(f\circ g)(u,w)&=B_{0}
\int_{\mathbb{R}}V_{P\overline{\phi}}^{A}\left\{P\overline{f}\right\}(m_{2},-m)\\
&\times V_{\psi}^{A}g(m_{3},m_{4})
e^{-i\frac{a}{2b}m(\frac{da}{2}-1)m_{4}}\rm{d}\mit m\\
		\end{split}
	\end{align}
where $P\overline{f}(x)=\overline{f(-x)}$, $m_{2}=u-\frac{a}{2}(w+m)$, $m_{3}=u+\frac{a}{2}m$, $B_{0}=\sqrt{i2\pi b}e^{i\frac{1}{b}(u-\frac{a}{2}w)(du_{0}-bw_{0})-i\frac{da}{2b}w(\frac{a}{4}w-u)-i\frac{a}{2b}(u^{2}+u^{2})}$\\
$m_{4}=w+m$
\end{thm}
\begin{proof}
Based on Definition 3 and Definition 4, we get
\begin{align}
		\begin{split}
	V_{\phi\circ\psi}^{A}(f\circ g)(u,w)&=\int_{\mathbb{R}}(f\circ g)(t)(\overline{\phi}\circ\overline{\psi})(t-w)K_{A}(t,u)\rm{d}\mit t\\
&=\int_{\mathbb{R}}\int_{\mathbb{R}}\overline{f(x)}g(x+t)e^{i\frac{a}{2b}x(t+x)}\rm{d}\mit x\\
&\times\int_{\mathbb{R}}\phi(r)\overline{\psi}(r+t-w)e^{i\frac{a}{2b}r(r+t-w)}\rm{d}\mit r\\
&\times K_{A}(t,u)\rm{d}\mit t\\
		\end{split}
	\end{align}
Setting $x=x_{1}, t=x_{2}-x_{1}, r=x_{1}-m$, we get
\begin{align}
		\begin{split}
	V_{\phi\circ\psi}^{A}(f\circ g)(u,w)&=\int_{\mathbb{R}^{3}}f(x_{1})g(x_{2})\phi(x_{1}-m)\\
&\times\overline{\psi}(x_{2}-m_{4})\\
&\times e^{i\frac{a}{b}x_{1}x_{2}-i\frac{a}{2b}x_{1}m_{4}-i\frac{a}{2b}mx_{2}+i\frac{a}{2b}mm_{4}}\\
&\times K_{A}(x_{2}-x_{1},u)\rm{d}\mit x_{1}\rm{d}\mit x_{2}\rm{d}\mit m\\
&=B_{0}\int_{\mathbb{R}}\int_{\mathbb{R}}\overline{f(x_{1})}\phi(x_{1}-m)\\
&\times K_{A}(x_{1},m_{2})\rm{d}\mit x_{1}\\
&\times\int_{\mathbb{R}}g(x_{2})\overline{\psi}(x_{2}-m_{4})K_{A}(x_{2},m_{3})\rm{d}\mit x_{2}\\
&\times e^{-i\frac{a}{2b}m(\frac{da}{2}-1)(m+w)}\rm{d}\mit m\\
&=B_{0}\int_{\mathbb{R}}V_{P\overline{\phi}}^{A}\left\{P\overline{f}\right\}(m_{2},-m)\\
&\times V_{\psi}^{A}g(m_{3},m_{4})
e^{-i\frac{a}{2b}m(\frac{da}{2}-1)m_{4}}\rm{d}\mit m\\
		\end{split}
	\end{align}
where $K_{A}(x_{1},m_{2})=\frac{1}{\sqrt{i2\pi b}}e^{i\frac{a}{2b}(-x_{1})^{2}-i\frac{1}{b}(-x_{1})}\\
\times e^{u-\frac{a}{2}(w+m)-u_{0}}\\
\times e^{-i\frac{1}{b}(u-\frac{a}{2}(w+m))(du_{0}-bw_{0})
+i\frac{d}{2b}((u-\frac{a}{2}(w+m))^{2}+u_{0}^{2})}$,\\
$K_{A}(x_{2},m_{3})=\frac{1}{\sqrt{i2\pi b}}e^{i\frac{a}{2b}x_{2}^{2}-i\frac{1}{b}x_{2}(u+\frac{a}{2}m-u_{0})}\\
\times e^{-i\frac{1}{b}(u+\frac{a}{2}m)(du_{0}-bw_{0})+i\frac{d}{2b}((u+\frac{a}{2}m)^{2}+u_{0}^{2})}$.  \end{proof}
\begin{corollary}
If the parameter of WOLCT changes to $(a,b,c,d,u_{0},w_{0})=(0,1,-1,0,0,0)$, then the Theorem 2 reduces to correlation theorem as follows:
\begin{align}
		\begin{split}
	V_{\phi\circ\psi}^{A}(f\circ g)(u,w)&=\sqrt{i2\pi b}e^{-iuw_{0}}
\int_{\mathbb{R}}V_{P\overline{\phi}}^{A}\left\{P\overline{f}\right\}(u,-m)\\
&\times V_{\psi}^{A}g(u,w+m)\rm{d}\mit m
		\end{split}
	\end{align}
\end{corollary}
\section{Conclusion}
In this paper, based on the association between the window function and the OLCT, we have studied the WOLCT. We first study some properties of the WOLCT, such as shift, modulation and orthogonality relation. Then, the convolution and correlation theorems for the WOLCT are shown. In our future works, the applications of the convolution and correlation theorems in signal and image processing will be investigated, and sampling theorem for the WOLCT and its applications will be studied.

\vspace{12pt}

\end{document}